\documentclass[12pt]{amsart}

\usepackage{amssymb}

\let\Bbb\mathbb

\let\mc\mathcal

\usepackage{eucal}
\let\euc\mathcal
\let\mathcal\mc

\def\>{\relax\ifmmode\mskip.666667\thinmuskip\relax\else\kern.111111em\fi}
\def\:{\relax\ifmmode\mskip.333333\thinmuskip\relax\else\kern.0555556em\fi}
\def\<{\relax\ifmmode\mskip-.333333\thinmuskip\relax\else\kern-.0555556em\fi}
\def\?{\relax\ifmmode\mskip-.666667\thinmuskip\relax\else\kern-.111111em\fi}
\def\vsk#1>{\vskip#1\baselineskip}
\def\vv#1>{\vadjust{\vsk#1>}\ignorespaces}
\def\vvn#1>{\vadjust{\nobreak\vsk#1>\nobreak}\ignorespaces}

 \let\vp\vphantom \let\hp\hphantom \let\alb\allowbreak

\newdimen\itemid
\def\pitem#1#2){\setbox0\hbox{#1)\enspace\kern\parindent}\itemid\wd0\par
 \hangindent\itemid\indent\hp{b}\llap{#2})\enspace\ignorespaces}

\let\al\alpha
\let\bt\beta
\let\dl\delta
\let\Dl\Delta
\let\eps\varepsilon

\let\la\lambda

\let\phi\varphi
\let\si\sigma

\let\Om\Omega

\let\ge\geqslant
\let\geq\geqslant
\let\le\leqslant

\let\der\partial

\let\ox\otimes

\let\bra\langle
\let\ket\rangle
\let\emptyset\varnothing

\let\on\operatorname
\let\bs\boldsymbol

\def\lsym#1{#1\alb\dots\relax#1\alb}
\def\lc{\lsym,}
\def\lox{\lsym\ox}
\def\ltimes{\lsym\times}

\def\diag{\on{diag\>}}
\def\End{\on{End\>}}
\def\id{\on{id}}

\def\tr{\on{tr}}

\def\GL{\on{GL}}
\def\Gr{\on{Gr}}

\def\C{{\mathbb C}}
\def\R{{\mathbb R}}
\def\Z{{\mathbb Z}}

\def\beq{\begin{equation}}
\def\eeq{\end{equation}}
\def\be{\begin{equation*}}
\def\ee{\end{equation*}}

\def\chii{\chi^{\vp1}}

\def\zbb{\bar{\bs z\:}\<}

\def\zbti{\tilde{\bs z\>}\?}
\def\labt{\tilde{\bs\la\<}\:}
\def\De{\euc D}

\def\Bc{\mc B}
\def\Fc{\mc F}
\def\Fcb{\Fc_{\<\bullet}}

\def\Ae{\euc A}
\let\Deb\De

\newcommand{\gl}{\mathfrak{gl}}

\newcommand\Ref{\eqref}

\newcommand{\bea}{\begin{eqnarray}}
\newcommand{\ena}{\end{eqnarray}}
\def\bel{\begin{eqnarray}}
\def\enl{\end{eqnarray}}

\numberwithin{equation}{section}

\newtheorem{thm}{Theorem}[section]
\newtheorem{cor}[thm]{Corollary}
\newtheorem{lem}[thm]{Lemma}
\newtheorem{prop}[thm]{Proposition}

\numberwithin{equation}{section}

\theoremstyle{definition}
\newtheorem*{rem}{Remark}
\newtheorem*{example}{Example}

\begin{document}

\hrule width0pt
\vsk->

\begin{title}[Lower bounds in real Schubert calculus]
{Lower bounds for numbers of real solutions\\[2pt]
in problems of Schubert calculus}
\end{title}
\author{E.\,Mukhin and V\<.\,Tarasov}
\address{EM: Department of Mathematical Sciences,
Indiana University\,--\>Purdue University\newline
\strut\kern\parindent\hp{ET: }
Indianapolis, 402 N.\,Blackford St., Indianapolis, IN 46202}
\email{mukhin@math.iupui.edu}
\address{VT: Department of Mathematical Sciences,
Indiana University\,--\>Purdue University\newline
\strut\kern\parindent\hp{VT: }
Indianapolis, 402 N.\,Blackford St., Indianapolis, IN 46202, \;and\newline
\strut\kern\parindent\hp{VT: }
St.\,Petersburg Branch of Steklov Mathematical Institute, Fontanka 27,
\newline
\strut\kern\parindent\hp{VT: }
St.\,Petersburg 191023, Russia}
\email{vt@math.iupui.edu, vt@pdmi.ras.ru}

\begin{abstract} We give lower bounds for the numbers of real solutions
in problems appearing in Schubert calculus in the Grassmannian $\Gr(n,d)$
related to osculating flags. It is known that such solutions are related to
Bethe vectors in the Gaudin model associated to \,$\gl_n$. The Gaudin Hamiltonians
are selfadjoint with respect to a nondegenerate indefinite Hermitian form.
Our bound comes from the computation of the signature of that form.
\end{abstract}

\maketitle

\thispagestyle{empty}

\section{Introduction}
It is well known that the problem of finding the number of real solutions
to algebraic systems is very difficult, and not many results are known.
In this paper we address the counting of real points in intersections of
Schubert varieties associated to osculating flags in the Grassmannian of
\,$n$-dimensional planes in a \,$d$-dimensional space. These problems are
parametrized by partitions $\la^{(1)}\lc\la^{(k)}$ and $\nu$
with at most $n$ parts satisfying the condition
\,$|\nu|+\sum_{i=1}^k|\la^{(i)}|=n\>(d-n)$, and distinct complex numbers
$z_1\lc z_k$. In this parametrization, $\la^{(1)}\lc\la^{(k)}$ and $\nu$
are respectively paired with $z_1\lc z_k$ and infinity.

\vsk.1>
Equivalently, we count $n$-dimensional real vector spaces of polynomials that
have ramification points $z_1\lc z_k$ with respective ramification conditions
\,$\la^{(1)}\lc\la^{(k)}$ and are spanned by polynomials of degrees
\,$d-i-\nu_{n+1-i}$, \;$i=1\lc n$, see Section \ref{schubert sec} for details.

\vsk.1>
The same number is obtained by counting real monic monodromy-free
Fuchsian differential operators with singular points $z_1\lc z_k$ and infinity,
exponents \,$\la^{(i)}_n\<,\la^{(i)}_{n-1}\?+1\lc\la^{(i)}_1\?+n-1$
\,at the points \,$z_i$\>, \,$i=1\lc k$, and exponents
\,$\nu_n\<+1-d\:,\nu_{n-1}\<+2-d\:\lc\nu_1\<+n-d$ \,at infinity.

\vsk.2>
The number of complex solutions to the above\:-mentioned algebraic systems
is readily given by the Schubert calculus and equals the multiplicity of
the irreducible \,$\gl_n$-module \,$L_\mu$ of highest weight
$\mu=(d-n-\nu_n,d-n-\nu_{n-1}\lc d-n-\nu_1)$ \>in the tensor product
$L_{\la^{(1)}}\!\ox\dots\ox L_{\la^{(k)}}$ of irreducible \,$\gl_n$-modules
of highest weights $\la^{(1)}\<\lc\la^{(k)}$.

\vsk.2>
The Shapiro\:-Shapiro conjecture proved in \cite{EG1} for \>$n=2$ \>and in
\cite{MTV4} for all \>$n$ \>asserts that if all $z_1\lc z_k$ are real, then
all solutions of the Schubert problem associated to osculating flags are real.
Therefore in this case, the number of real solutions is maximal possible.

\vsk.1>
Next we wonder how many real solutions we can guarantee in other cases.
Clearly for the Schubert problem to have real solutions, the set $z_1\lc z_k$
should be invariant under the complex conjugation and the ramification
conditions at the complex conjugated points should be the same. In this case
we say that the data \,$z_1\lc z_k$, $\la^{(1)}\lc\la^{(k)}$ are invariant
under the complex conjugation. In general, the number of real solutions
is not known, and based on extensive computer experimentation, see \cite{HS},
the answer to this question should be very interesting.

\vsk.2>
Prior to this paper, there were several approaches to obtaining lower bounds.
First, one can compute the real topological degree of the Wronski map, and it
gives bounds for the case when all \>$\la^{(1)}\<\lc\la^{(k)}\<$ are one\:-box
partitions, see \cite{EG2}. The lower bound can be extended to the case when
all but one partitions consist of one box, see \cite{SS}. While this method
gives nontrivial bounds, it has several serious drawbacks --- the answer does
not depend on the number of real points among \> $z_1\lc z_k$, does not apply
to general ramification conditions, and is far from being sharp in many cases.

\vsk.1>
Another method is to consider parity conditions. It is proved in \cite{HSZ}
that if all partitions are symmetric, the number of solutions can change only
by $4$. Unfortunately, this is also a very special situation and the only lower
bound one can obtain this way is $2$. Finally, in some cases, see Theorem~7
in \cite{HHS}, the required spaces of polynomials can be described relatively
explicitly to estimate the number of solutions. This estimate is sharp, that
is, it is attained for some choice of \> $z_1\lc z_k$, but it works only in
very special situations and cannot be possibly extended.

\vsk.2>
We propose one more way to attack the problem. The proof of the
Shapiro\:-Shapiro conjecture in \cite{MTV3}, \cite{MTV4} is based on the
identification of the spaces of polynomials with points of spectrum of
a remarkable family of commuting linear operators known as higher Gaudin
Hamiltonians. For real \>$z_1\lc z_k$, these operators are selfadjoint with
respect to a positive definite Hermitian form, and hence have real eigenvalues.
Eventually, this shows that the spaces of polynomials with real ramification
points are real.

\vsk.1>
If some of \>$z_1\lc z_k$ are not real, but the data $z_1\lc z_k$,
\>$\la^{(1)}\<\lc\la^{(k)}$ are invariant under the complex conjugation,
the higher Gaudin Hamiltonians are selfadjoint with respect to a nondegenerate
Hermitian form, but this form is indefinite. Since the number of real
eigenvalues of such operators is at least the absolute value of the signature
of the Hermitian form, see Lemma~\ref{linalg}, this gives a lower bound
for the number of real solutions to the Schubert problem in question.

\vsk.1>
We reduce the computation of the signature of the form to the computation of
values of characters of products of symmetric groups on products of commuting
transpositions. There is a formula for such characters, see
Proposition~\ref{char prop}, similar to the Frobenius formula~\cite{F}.
Thus, we obtain a lower bound for all possible choices of partitions
\>$\la^{(1)}\<\lc\la^{(k)}$ and \>$\nu$, and the obtained bound depends on
the number of real points among \>$z_1\lc z_k$, see Corollary~\ref{bound thm}.

\vsk.2>
We check the obtained lower bound against the available results and computer
experiments, see Section~\ref{compare sec}. We find that our bound is sharp
in many cases. For example, all available data for $n=2$ match our bound.
However, our bound is not sharp in general. We hope that the bound can be
improved in some cases by modifying the Hermitian form given in this paper so
that higher Gaudin Hamiltonians remain selfadjoint relative to the new form.

\vsk.2>
The paper is organized as follows. We start with computations of characters
of symmetric groups in Section \ref{char sec}, see Proposition \ref{char prop}.
Then we prepare notation and definitions for osculating Schubert calculus in
Section \ref{schubert sec}. We recall definitions and properties of higher
Gaudin Hamiltonians in Section \ref{gaudin sec} and their symmetries in
Section \ref{shap sec}. We discuss the key facts from linear algebra
about selfadjoint operators with respect to indefinite Hermitian form in
Section \ref{lin alg sec}. In Section \ref{bound sec} we prove our main
statement, see Theorem \ref{sgn} and Corollary \ref{bound thm}.
In Section \ref{compare sec} we compare our bounds with known data and results.

\newpage
\section{Characters of the symmetric groups}\label{char sec}
The study of characters of the symmetric groups is a classical subject which
goes back to Frobenius \cite{F}. In this section we deduce a formula for
characters of a product of the symmetric groups appearing in a tensor
product of irreducible \,$\gl_n$-modules.

\vsk.2>
Let $S_k$ be the group of all permutations of a $k$-element set,
\,$\GL_n$ be the group of all nondegenerate $n\times n$ matrices, and
\,$\gl_n$ be the Lie algebra of $n\times n$ matrices.

\vsk.1>
Let $\la=(\la_1,\la_2\lc\la_n)$ be a partition with at most $n$ parts,
\vvn.1>
$\la_1\geq \la_2\geq \dots\geq \la_n\geq 0$. We use the notation
$|\la|=\sum_{i=1}^n\la_i$.

\vsk.1>
For each partition $\la$ with at most $n$ parts, denote by $L_\la$ the
irreducible finite-dimensional \,$\gl_n$-module of highest weight $\la$.
We call the module corresponding to $\la=(1,0\lc 0)$ the vector representation.

Let
\vvn-.4>
\beq
\label{Delta} \Dl_n=\prod_{i,j=1,\ i>j}^n(x_i-x_j)=\det(x_i^{n-j})_{i,j=1}^n\in\C[x_1\lc x_n]
\vv.2>
\eeq
be the Vandermonde determinant.
\vvn.5>
Let $S_\la\in\C[x_1\lc x_n]$ be the Schur polynomial given by
\beq
\label{Schur}
S_\la(x_1\lc x_n)=\frac{\det(x_i^{\la_j+n-j})_{i,j=1}^n}{\Dl_n}.
\vv.3>
\eeq
The Schur polynomial is a symmetric polynomial in \>$x_1\lc x_n$.
It is well known that the character of the module \>$L_\la$ is given
by the Schur polynomial:
\vvn.2>
\be
S_\la(x_1\lc x_n)=\tr_{L_\la}\!X,
\ee
where $X=\diag(x_1\lc x_n)\in {\GL}_n$.

\vsk.2>
Consider the tensor product of \,$\gl_n\<$-modules:
\vvn.2>
\beq
\label{tensor prod}
L_{\bs\la}=L_{\la^{(1)}}^{\ox k_1}\ox L_{\la^{(2)}}^{\ox k_2}\lox
L_{\la^{(s)}}^{\ox k_s}
\vv.2>
\eeq
and its decomposition into irreducible \,$\gl_n\<$-submodules:
\vvn.2>
\beq
\label{Mlm}
L_{\bs\la}=\,\bigoplus_\mu\,L_\mu\ox M_{\bs\la,\:\mu}\,.
\vv-.2>
\eeq
Notice that the multiplicity space $M_{\bs\la,\:\mu}$ is trivial unless
\vvn-.2>
\beq
\label{weight cond}
|\:\mu\:|\,=\,\sum_{i=1}^s\,k_i\>|\la^{(i)}|\,.
\eeq
The product of symmetric groups
\;$S_{\bs k}=S_{k_1}\!\times S_{k_2}\!\ltimes S_{k_s}$
acts on $L_{\bs\la}$ by permuting the corresponding tensor factors.
Since the \,$S_{\bs k}$-action commutes with the \,$\gl_n\<$-action, the group
$S_{\bs k}$ acts on the multiplicity space $M_{\bs\la,\:\mu}$ for all $\mu$.
If $s=1$ and all tensor factors are vector representations,
$\la^{(1)}=(1,0\lc 0)$, by the Schur\:-Weyl duality, the space
$M_{\bs\la,\:\mu}$ is the irreducible representation of \,$S_{k_1}\!$
corresponding to the partition \>$\mu$. In general, $M_{\bs\la,\:\mu}$ is
a reducible representation of $S_{\bs k}$.

\vsk.2>
For \,$\si=\si_1\times\si_2\times\dots\times\si_s\in S_{\bs k}$,
\;$\si_i\in S_{k_i}$, let
\,$\chii_{\bs\la,\:\mu}(\si)=\tr_{M_{\bs\la,\:\mu}}\!\si$
\vvn.1>
\,be the value of the character of \,$S_{\bs k}$ corresponding to
the representation $M_{\bs\la,\:\mu}$ on \,$\si$. Writing \,$\si_i$ as
a product of disjoint cycles, denote by \>$c_i$ the number of cycles in
the product and by \,$l_{ij}$, \,$j=1\lc c_i$, the lengths of cycles.
We have \,$l_{i,1}\lsym+l_{i,c_i}\<=\:k_i$.

\begin{prop}\label{char prop}
The character value \,$\chii_{\bs\la,\:\mu}(\si)$ equals the coefficient
of the monomial\\ $x_1^{\:\mu_1+n-1}x_2^{\:\mu_2+n-2}\?\dots x_n^{\:\mu_n}$
in the polynomial
\be
\Dl_n\cdot\>\prod_{i=1}^s\,\prod_{j=1}^{c_i}\,
S_{\la^{(i)}}(x_1^{l_{ij}}\?\lc x_n^{l_{ij}})\,.
\ee
\end{prop}
\begin{proof}
Let \>$V$ be a vector space, \,$P\in\End(V\?\ox V)$ be the flip map,
and \,$A,B\in\End(V)$.
Then \,$(\id\ox\tr_V)\>\bigl((A\ox B)\>P\:\bigr)=A\:B\in\End(V)$\,.

Let $\si=(1\:2\dots l)$ be a cycle permutation and
\,$X=\diag(x_1\lc x_n)\in\GL_n$\>. Using the presentation
$\si=(1\:2)\>(2\:3)\dots(\:l-1\;l\:)$\>, \,we get
\beq
\label{cycle acts}
\tr_{L_\la^{\ox l}}(X\<\times\si)\,=\,\tr_{L_\la}(X^l)\,=\,
S_\la(x_1^l\lc x_n^l)\,.
\eeq

For any \,$\si\in S_{\bs k}$ and \>$X\<\in\GL_n$, formulae \Ref{tensor prod}
and \Ref{cycle acts} yield
\be
\tr_{L_{\bs\la}}(X\<\times\si)\,=\,
\prod_{i=1}^s\,\prod_{j=1}^{c_i}\,S_{\la^{(i)}}(x_1^{l_{ij}}\lc x_n^{l_{ij}})\,,
\vv-.2>
\ee
and formulae \Ref{Mlm} and \Ref{Schur} give
\be
\tr_{L_{\bs\la}}(X\<\times\si)\,=\,\sum_\mu\,
\chii_{\bs\la,\:\mu}(\si)\,S_\mu(x_1\lc x_n)\,=\,
\frac1{\Dl_n\!}\>\sum_\mu\,\chii_{\bs\la,\:\mu}(\si)\,
\det(x_i^{\mu_j+n-j})_{i,j=1}^n\,.
\ee
The proposition follows.
\end{proof}

For the case of vector representations: $s=1$, \,$\la^{(1)}=(1,0\lc 0)$,
the Schur polynomial is
\,$S_{\la^{(1)}}(x_1\lc x_n)=x_1+x_2+\dots+x_n$ and Proposition~\ref{char prop}
reduces to the famous Frobenius formula \cite{F} for characters of irreducible
representations of the symmetric group.

\section{Osculating Schubert calculus}\label{schubert sec}
In this section we recall the problem of computing intersections of Schubert
varieties corresponding to osculating flags.

\vsk.2>
Let $n,d$ be positive integers such that $d>n$. Let \>$V$ be
a \,$d$-dimensional complex vector space. We realize \>$V$ as the space of
polynomials in a variable $x$ of degree less than $d$: \>$V=\C_d[x]$.
The Grassmannian $\Gr(n,d)$ of \,$n$-dimensional planes in \>$V$ is a smooth
projective variety of dimension $n(d-n)$.

\vsk.2>
For $z\in\C$ we define a full flag $\Fcb(z)$ in \>$V$ as follows:
\vvn.3>
\be
\Fcb(z)\,=\,\{\>\Fc_1(z)\subset\Fc_2(z)\subset\ldots\subset\Fc_{d-1}(z)\subset
\Fc_d(z)=V\:\}\,,
\vv.3>
\ee
where \,$\Fc_i(z)=(x-z)^{d-i}\>\C_i[x]$ \,is the subspace of polynomials
vanishing at $z$ to the order at least \,$d-i$\>. Clearly, \,$\Fc_i(z)$ \>has
a basis \,$(x-z)^{d-i}\lc(x-z)^{d-1}$ \,and \,$\dim\Fc_i(z)=i$\,.
We also define a full flag
$\Fcb(\infty)\,=\,\{\>\Fc_1(\infty)\subset\Fc_2(\infty)\subset\dots\subset
\Fc_{d-1}(\infty)\subset\Fc_d(\infty)=V\:\}$\,, \,where
\,$\Fc_i(\infty)\>=\>\C_i[x]$ is the subspace of polynomials of degree less
than $i$. The subspace $\Fc_i(\infty)$ has a basis \,$1,x\lc x^{i-1}$.

\vsk.2>
Given \,$z\in\C\cup\{\infty\}$ \,and a partition \>$\la$ \>with at most $n$
parts, the corresponding Schubert variety is
\be
\Om_\la(z)\,=\,\{\>W\in\Gr(n,d)\ |\ \dim W\cap\Fc_{d-\la_{n-i}-i}(z)\geq n-i\,,
\quad i=0\lc n-1\>\}\,.
\vv.2>
\ee
The Schubert variety \,$\Om_\la(z)\subset\Gr(n,d)$ has codimension \>$|\la|$\>.

\vsk.2>
Given partitions \>$\la^{(1)}\<\lc\la^{(k)}$ and \>$\nu$ with at most $n$ parts
such that
\beq
\label{nu}
|\nu|+\sum_{i=1}^k |\la^{(i)}|\,=\,n(d-n)\,,
\eeq
and distinct complex numbers \>$z_1\lc z_k$\>, the corresponding osculating
\vvn.1>
Schubert problem asks to find the intersection of Schubert varieties
\vvn.2>
\beq
\label{Omlnz}
\Om(\bs\la\:,\nu,\bs z)\,=\,
\bigcap_{i=1}^k\;\Om_{\la^{(i)}}(z_i)\,\cap\,\Om_{\nu}(\infty)\,.
\eeq
This intersection consists of \,$n$-dimensional spaces of polynomials
\>$W\<\subset V$ \>such that

\pitem ba)
the space \>$W$ has a basis $f_{1,0}\lc f_{n,0}$ \>such that
$\deg f_{j,0}=d-i-\nu_{n+1-i}$\>, \>and

\pitem bb)
for each $i=1\lc k$\>, the space \>$W$ has a basis $f_{1,i}\lc f_{n,i}$ such
that $f_{j,i}$ has a root at \>$z_i$ of order exactly \>$\la_{n+1-j}+j-1$\>.

\vsk.1>
According to Schubert calculus, the set \,$\Om(\bs\la\:,\nu,\bs z)$
\>is finite, and the number \,$m(\bs\la\:,\nu)$ of complex points
in \,$\Om(\bs\la\:,\nu,\bs z)$ \>counted with multiplicities equals
the multiplicity of the irreducible \,$\gl_n$-module \>$L_\mu$ in the tensor
product \>$L_{\la^{(1)}}\lox L_{\la^{(k)}}$, where the partition \>$\mu$ is
the complement of \>$\nu$ in the $n\times (d-n)$ rectangle:
\vvn.3>
\beq
\label{mu nu}
\mu\,=\,(d-n-\nu_n\>,\>d-n-\nu_{n-1}\>\lc\>d-n-\nu_1)\,.
\vv.2>
\eeq

It is known that for generic complex \>$z_1\lc z_k$\>, all points of
intersection are multiplicity-free. Moreover, for distinct real
\>$z_1\lc z_k$\>, all points of intersection are multiplicity-free as well,
and all the corresponding spaces of polynomials are real, see \cite{MTV4}.
That is, for distinct real $z_1\lc z_k$ the osculating Schubert problem has
\,$m(\bs\la\:,\nu)$ \>real solutions.

\vsk.2>
Let us make two pertinent remarks. First, notice that
\,$m(\bs\la\:,\nu)=m(\labt,\emptyset)$\>, where \,$\labt$
\>is the \,$(k+1)$-tuple \,$\la^{(1)}\lc\la^{(k)},\:\nu$ and
\,$\emptyset=(0\lc 0)$ \>is the empty partition.

\vsk.1>
Second, fix partitions \,$\la^{(1)}\<\lc\la^{(k)}$ and \>$\mu$ such that
\vvn.1>
\,$|\:\mu\:|=\sum_{i=1}^k|\la^{(i)}|$\,, take \,$d\ge n+\mu_1$, \>and set
\vvn-.2>
\beq
\label{numu}
\nu\>=\,(d-n-\mu_n\>\lc\>d-n-\mu_1)\,.
\vv.2>
\eeq
Then the spaces of polynomials that are points of \,$\Om(\bs\la\:,\nu,\bs z)$
\>do not depend on \,$d$.

\section{Gaudin model}
\label{gaudin sec}

{\baselineskip1.1\baselineskip
Let \>$E_{ij}$\>, \,$i,j=1\lc n$\>, \,be the standard basis of \,$\gl_n$\>:
\,$[\:E_{ij}\:,E_{k\:l}]\:=\:\dl_{jk}\>E_{i\:l}-\dl_{i\:l}\>E_{kj}$\,.
The current Lie algebra \,$\gl_n[t]$ \,is spanned by the elements
$E_{ij}\<\ox t^{\:r}$, \,$i,j=1\lc n$\>, \,$r\in\Z_{\geq 0}$\>,
\,satisfying the relations
\,$[\:E_{ij}\<\ox t^{\:r}\<,E_{k\:l}\<\ox t^s\:]\:=\:
\dl_{jk}\>E_{i\:l}\<\ox t^{\:r+s}\?-\dl_{i\:l}\>E_{kj}\<\ox t^{\:r+s}$.
We identify \,$\gl_n$ with the subalgebra in \,$\gl_n[t]$ by the rule
$E_{ij}\mapsto E_{ij}\<\ox 1$\>, \,$i,j=1\lc n$.

\vsk.1>
Given \>$z\in\C$\>, define the evaluation homomorphism
\;$\eps_z\?:\gl_n[t]\to\:\gl_n$\>,
\,$E_{ij}\<\ox t^{\:r}\<\mapsto E_{ij}\>z^r$.
For a \,$\gl_n$-module \>$L$\>, the evaluation \,$\gl_n[t]$-module \>$L(z)$
is the pull-back of \>$L$ through the evaluation homomorphism \,$\eps_z$.

For \,$g\in\gl_n$\>, define the formal power series in $x^{-1}$:
$g(x)=\sum_{s=0}^\infty (\:g\ox t^s)\>x^{-s-1}$. The series $g(x)$ acts
in the evaluation module \>$L(z)$ as \,$g\>(x-z)^{-1}$.
\par}

\vsk.2>
Let \,$\der_x$ be the differentiation with respect to \>$x$.
Set \,$X_{ij}=\:\dl_{ij}\>\der_x-\:E_{ij}(x)$, \,$i,j=1\lc n$\>.
Define the formal differential operator \,$\Deb$ \,by the rule
\beq
\label{gen}
\Deb\,=\>\sum_{\si\in S_n\!}\,
X_{\si(1),\:1}\>X_{\si(2),\>2}\dots X_{\si(n),\>n}\,=\,
\der_x^{\>n}+\sum_{i=1}^n\,\sum_{j=i}^\infty\,B_{ij}\,x^{-j}\,\der_x^{\>n-i}\,,
\eeq
where \,$B_{ij}$ \,are elements of the universal enveloping algebra
\,$U(\:\gl_n[t])$. The operator \,$\Deb$ \,is called the universal operator.

\vsk.2>
The unital subalgebra of \,$U(\gl_n[t])$ \,generated by \>$B_{ij}$,
\,$i=1\lc n$, \,$j\in\Z_{\geq\:i}$\>, \,is called the Bethe subalgebra and
denoted by \,$\Bc_n$. Also, \,$\Bc_n$ \>is called the algebra of higher
Gaudin Hamiltonians.

\begin{prop}[\cite{T}]
\label{bethe}
The subalgebra \,$\Bc_n$ is commutative and commutes with \,$\gl_n$.
\qed
\end{prop}

For partitions \,$\la^{(1)}\<\lc\la^{(k)}$ and distinct complex numbers
$z_1\lc z_k$\>, \>consider the tensor product
\,$L_{\bs\la}(\bs z)=L_{\la^{(1)}}(z_1)\lox L_{\la^{(k)}}(z_k)$ \>of evaluation
\vvn.2>
\,$\gl_n[t]$-modules. For every \,$g\in\gl_n$\>, the series \,$g(x)$ \,acts
on $L_{\bs\la}(\bs z)$ \>as a rational function of $x$.

\vsk.1>
As a \,$\gl_n$-module, \>$L_{\bs\la}(\bs z)$ \>does not depend on $z_1\lc z_k$
\vvn.1>
\>and equals \,$L_{\bs\la}=L_{\la^{(1)}}\lox L_{\la^{(k)}}$\>.
Let \>$L_{\bs\la}=\bigoplus_\mu L_\mu\ox M_{\bs\la,\:\mu}$ \>be
its decomposition into irreducible \,$\gl_n$-submodules.
\vvn.1>
Recall that the multiplicity space \>$M_{\bs\la,\:\mu}$ is trivial unless
\vvn-.3>
\beq
\label{mula}
|\:\mu\:|\,=\,\sum_{i=1}^k|\la^{(i)}|\,.
\eeq

As a subalgebra of \>$U(\gl_n[t])$, the algebra \>$\Bc_n$ acts on
$L_{\bs\la}(\bs z)$. Since $\Bc_n$ commutes with \,$\gl_n$, this action
descends to the action of \>$\Bc_n$ on each multiplicity space
\>$M_{\bs\la,\:\mu}$\>. For \,$b\in\Bc_n$\>, denote by
\,$b(\bs\la,\mu,\bs z)\in\End(M_{\bs\la,\:\mu})$ the corresponding linear
operator.

\vsk.1>
Given a common eigenvector \,$v\in M_{\bs\la,\:\mu}$ \,of the operators
\,$b(\bs\la,\mu,\bs z)$, denote by \,$b(\bs\la,\mu,\bs z;v)$
\>the corresponding eigenvalues, and define the scalar differential operator
\vvn.2>
\be
\De_v\>=\,\der_x^{\>n}+\sum_{i=1}^n\,\sum_{j=i}^\infty\,
B_{ij}(\bs\la,\mu,\bs z; v)\,x^{-j}\,\der_x^{\>n-i}\,.
\ee
One can check that \,$\De_v$ \>is a Fuchsian differential operator
with singular points at the points $z_1\lc z_k$ \>and infinity.
Moreover, for every \,$i=1\lc k$, the exponents of \,$\De_v$ at the point
\,$z_i$ \,are \,$\la^{(i)}_n\<,\la^{(i)}_{n-1}\?+1\lc\la^{(i)}_1\?+n-1$,
the exponents of \,$\De_v$ at infinity are
\,$-\:\mu_1\<+1-n,-\:\mu_2\<+2-n\lc-\:\mu_n$,
and the kernel of \,$\De_v$ is spanned by polynomials, see \cite{MTV2}.

\vsk.2>
Theorem \ref{langlands} below connects Schubert calculus and the Gaudin model.
Let a partition \>$\mu$ satisfy \Ref{mula}. Take \,$d\ge n+\mu_1$\>, \>and
define the partition \,$\nu$ \>by \Ref{numu}. Let \,$\Om(\bs\la,\nu,\bs z)$
\>be the intersection of Schubert varieties \Ref{Omlnz}.

\begin{thm}\cite{MTV4}
\label{langlands}
There is a bijective correspondence \,$\tau$ between common eigenvectors
of the operators \;$b(\bs\la,\mu,\bs z)\in\End(M_{\bs\la,\:\mu})$,
\;$b\in\Bc_n$\>, and points of \;$\Om(\bs\la,\nu,\bs z)$ such that
\,$\tau(v)$ \>is the kernel of the scalar differential operator \,$\De_v$\>.
For generic \,$\bs z$, \,the operators \;$b(\bs\la,\mu,\bs z)$ \,are
diagonalizable and have simple joint spectrum.
\qed
\end{thm}

\begin{rem}
Denote by \,$\Bc_n(\bs\la,\mu,\bs z)\subset\End(M_{\bs\la,\:\mu})$ \>the
commutative subalgebra, generated by the operators \;$b(\bs\la,\mu,\bs z)$\>,
\;$b\in\Bc_n$\>. It is proved in \cite{MTV4} that for all
\,$\bs z=(z_1\lc z_k)$ \,with distinct coordinates, $\Bc_n(\bs\la,\mu,\bs z)$
\>is a maximal commutative subalgebra of dimension \,$\dim M_{\bs\la,\:\mu}$\>,
\>and for a generic vector \,$w\in M_{\bs\la,\:\mu}$\>, the map
\be
\Bc_n(\bs\la,\mu,\bs z)\,\to\,M_{\bs\la,\:\mu}\,,\qquad
b(\bs\la,\mu,\bs z)\,\mapsto b(\bs\la,\mu,\bs z)\,w\,,
\ee
is an isomorphism of vector spaces.
\end{rem}

\section{Shapovalov form}
\label{shap sec}

For any partition \>$\la$ \>with at most $n$ parts, the irreducible
\,$\gl_n$-module \>$L_\la$ admits a positive definite Hermitian form
\,$(\cdot,\cdot)_\la$ such that \,$(E_{ij}\>v,w)_\la=(v,E_{ji}\>w)_\la$ \>for
any \,$i,j=1\lc n$ \>and any $v,w\in L_\la$\>. Such a form is unique up to
multiplication by a positive real number. We will call this form the Shapovalov
form.

\vsk.1>
For partitions \,$\la^{(1)}\<\lc\la^{(k)}$ we define the positive definite
Hermitian form \,$(\cdot,\cdot)_{\bs\la}$ on the tensor product
\,$L_{\bs\la}=L_{\la^{(1)}}\lox L_{\la^{(k)}}$ \>as the product of Shapovalov
forms on the tensor factors. For each multiplicity space
\>$M_{\bs\la,\:\mu}$\>, the form \,$(\cdot,\cdot)_{\bs\la}$ induces
a positive definite Hermitian form \,$(\cdot,\cdot)_{\bs\la,\:\mu}$ on
\>$M_{\bs\la,\:\mu}$\>.

\begin{prop}
\label{sym bethe}
For any \,$i=1\lc n$, \,$j\in\Z_{\geq i}$\>, and any
\,$v,w\in M_{\bs\la,\:\mu}$\>,
\vvn.2>
\beq
\bigl(B_{ij}(\bs\la,\mu,\bs z)\>v,w\bigr)_{\bs\la,\:\mu}\,=\,
\bigl(v,B_{ij}(\bs\la,\mu,\zbb)\>w\bigr)_{\bs\la,\:\mu}\,,
\vv.1>
\eeq
where \,$B_{ij}$ are defined by \Ref{gen}, \,$\zbb=(\bar z_1\lc\bar z_k)$ and
the bar stands for the complex conjugation.
\end{prop}
\begin{proof}
The claim follows from \cite[Theorem~9.1]{MTV1}.
\end{proof}

If some of the partitions \,$\la^{(1)}\lc\la^{(k)}$ coincide, the operators
\,$b(\bs\la,\mu,\bs z)$ have additional symmetry. Assume that
\,$\la^{(i)}=\la^{(i+1)}$ for some \,$i$\>. Let \>$P_i\in\End(L_{\bs\la})$
be the flip of the \,$i$-th and \,$(i+1)$-st tensor factors and
\,$\zbti^{(i)}=(z_1\lc z_{i-1},z_{i+1},z_i,z_{i+2}\lc z_k)$\>.

\begin{lem}
\label{perm bethe}
For any \,$b\in\Bc_n$\>, we have
\,$P_i\,b(\bs\la,\mu,\bs z)\>P_i\,=\,b(\bs\la,\mu,\zbti^{(i)})$\,.
\qed
\end{lem}

\section{Selfadjoint operators with respect to indefinite Hermitian form}
\label{lin alg sec}
In this section we discuss the key statements from linear algebra.

\vsk.2>
Given a finite-dimensional vector space $M$, a linear operator $A\in\End{M}$,
and a number $\al\in\C$, let \,$M_A(\al)=\ker\:(A-\al)^{\dim M}$.
When \,$M_A(\al)$ is not trivial, it is the subspace of generalized
eigenvectors of $A$ with eigenvalue $\al$.

\begin{lem}\label{linalg}
Let \>$M$ \?be a complex finite-dimensional vector space with a nondegenerate
Hermitian form of signature $m$, and let \>$A$ be a selfadjoint operator.
Let \>$R=\bigoplus_{\al\in\R}M_A(\al)$ be the subspace of generalized
eigenvectors of $A$ with real eigenvalues. Then the restriction of
the Hermitian form on \>$R$ \?is nondegenerate and has signature \>$m$.
In particular, \>$\dim\:R\ge|\:m\:|$\>.
\end{lem}
\begin{proof}
Since $A$ is selfadjoint, $M_A(\al)^\perp\?=\bigoplus_{\bt\ne\bar\al}M_A(\bt)$.
In particular, if $\al$ is an eigenvalue of $A$ that is not real,
the restriction of the Hermitian form on the subspace
\,$M_A(\al)^\perp\oplus M_A(\bar\al)$ is nondegenerate and has zero signature.
Thus, the restriction of the Hermitian form on the subspace $R$
\,is nondegenerate and has signature \,$m$\>.
\end{proof}

\begin{cor}
\label{linalg2}
Let \>$M$ \?be a complex finite-dimensional vector space with a nondegenerate
Hermitian form of signature \>$m$, and let \,$\Ae\subset\End(M)$ be
a commutative subalgebra over \,$\R$, whose elements are selfadjoint operators.
Let \>$R=\bigcap_{A\in\Ae}\bigoplus_{\al\in\R}M_A(\al)$.
Then the restriction of the Hermitian form on \>$R$ \?is nondegenerate
and has signature \>$m$. In particular, \>$\dim\:R\ge|\:m\:|$\>.
\end{cor}
\begin{proof}
Let \>$A_1\lc A_k$ be a basis of \>$\Ae$. Clearly,
\vvn.1>
$R=\bigcap_{\>i=1}^{\>k}\bigoplus_{\al\in\R} M_A(\al)$.
Let \>$M_1=\bigoplus_{\al\in\R}M_{A_1}(\al)$. The subspace \>$M_1$ is
\vv.1>
\>$\Ae$-invariant and the restriction of the Hermitian form on \>$M_1$ is
nondegenerate and has signature \>$m$ by Lemma \ref{linalg}.
The corollary follows by induction.
\end{proof}

In fact, Lemma \ref{linalg} can be strengthened.

\begin{lem}[\cite{P}]
\label{pontr lem}
Under the assumption of Lemma \ref{linalg}, the operator \>$A$ has at least
\>$m$ linearly independent eigenvectors with real eigenvalues:
$\dim\:\bigoplus_{\al\in\R}\ker\:(A-\al)\geq m$.
\qed
\end{lem}

Contrary to the case of positive definite Hermitian form, Lemma \ref{pontr lem}
does not extend to a pair of commuting selfadjoint operators. A counterexample
is given by the multiplication operators in the ring $\C[x,y]/(x^2\?=y^2,xy=0)$
with the usual Grothendieck residue form. Explicitly, we have
a four-\:dimensional commutative real unital algebra of linear operators
in \,$\C^4$ generated by two matrices
\vvn.1>
\be
x\,=\left(\begin{matrix} 0 & 0 & 0 & 0 \\ 1 & 0 & 0 & 0 \\ 0 & 1 & 0 & 0 \\
0 & 0 & 0 & 0 \end{matrix}\right), \qquad
y\,=\left(\begin{matrix} 0 & 0 & 0 & 0 \\ 0 & 0 & 0 & 0 \\ 0 & 0 & 0 & 1 \\
1 & 0 & 0 & 0 \end{matrix}\right),
\vv.1>
\ee
that satisfy the relations $x^2\?=y^2$, \,$x^3\?=y^3\?=xy=yx=0$.
In particular, both \,$x$ and \,$y$ have the only eigenvalue that equals zero:
$M=M_x(0)=M_y(0)$. Clearly, \;$\dim\:\ker\:x=\dim\:\ker\:y=2$ \,and
\;$\dim\:(\ker x\bigcap\ker y)=1$\>.

\vsk.2>
The Hermitian form is given by the matrix
\vvn.2>
\be
J=\left(\begin{matrix} 0 & 0 & 1 & 0 \\ 0 & 1 & 0 & 0 \\ 1 & 0 & 0 & 0 \\
0 & 0 & 0 & 1 \end{matrix}\right).
\vv.1>
\ee
It is nondegenerate and has signature two. Since \,$x^t\<J=J\:\bar x$
\,and \,$y^t\<J=J\:\bar y$, the operators \>$x$ \>and \>$y$ \>are selfadjoint
and commuting, but have only one common eigenvector.

\vsk.2>
The given counterexample is minimal. If in addition to the assumption
of Corollary \ref{linalg2}, for each character \,$\rho:\Ae\to\C$ \,we have
\vvn.1>
\,$\dim\:\bigcap_{A\in\Ae}M_A(\rho(A))<4$\>, then there are at least \>$m$
\>linearly independent common eigenvectors of the elements of \>$\Ae$ with
\vvn.1>
real eigenvalues,
\;$\dim\:\bigcap_{A\in\Ae}\bigoplus_{\al\in\R}\ker(A-\al)\geq m$\>.

\section{The lower bound}
\label{bound sec}
In this section we prove our main theorem --- the lower bound for the number
of real solutions to osculating Schubert problems, see Theorem \ref{sgn} and
Corollary \ref{bound thm}.

\vsk.2>
Recall the notation from Section \ref{schubert sec}. For positive integers
$n,d$ such that $d>n$ we consider the Grassmannian of $\Gr(n,d)$ of
\,$n$-dimensional planes in the space \,$\C_d[x]$ \,of polynomials of degree
less than $d$. A point $W\in\Gr(n,d)$ is called real if it has a basis
consisting of polynomials with real coefficients.

\vsk.2>
Given partitions \,$\la^{(1)}\lc\la^{(k)}$ and $\nu$ with at most $n$ parts
satisfying \Ref{nu}\:, and distinct complex numbers \,$z_1\lc z_k$\>, denote
by \,$d(\bs\la,\nu,\bs z)$ the number of real points counted with
multiplicities in the intersection of Schubert varieties
\,$\Om(\bs\la,\nu,\bs z)\subset\Gr(n,d)$. Clearly, $d(\bs\la,\nu,\bs z)=0$
\,unless the set \,$\{z_1\lc z_k\}$ \,is invariant under the complex
conjugation and \>$\la^{(i)}\?=\la^{(j)}$ whenever $z_i=\bar z_j$. In what
follows we denote by \,$c$ \,the number of complex conjugate pairs in the set
\,$\{z_1\lc z_k\}$ and without loss of generality assume that
\,$z_1=\bar z_2\,,\lc z_{2c-1}=\bar z_{2c}$
\,while \,$z_{2c+1}\lc z_k$ are real. We will also always assume that
\,$\la^{(1)}\?=\la^{(2)}\<\lc\la^{(2c-1)}\?=\la^{(2c)}$.

\vsk.2>
For the sake of clarity, let us emphasize that by {\it generic\/} we always
mean {\it on a nonempty Zariski open subset of\/} $\C^k$. Recall that for any
\,$\bs\la,\nu$ \>and generic complex \,$\bs z$, the intersection of Schubert
varieties is transversal, that is, all points of \,$\Om(\bs\la,\nu,\bs z)$
are multiplicity-free. The same holds true under the reality condition
on \,$\bs z,\bs\la$ \>imposed above for any \>$c$\>.

\vsk.2>
Let \,$L_{\bs\la}=L_{\la^{(1)}}\lox L_{\la^{(k)}}$ \>be the tensor product
of irreducible \,$\gl_n\<$-modules and let \>$M_{\bs\la,\:\mu}$
\>be the multiplicity space of \>$L_\mu$ in \,$L_{\bs\la}$, see
Section \ref{gaudin sec}. Since \,$\la^{(2i-1)}\?=\la^{(2i)}$ for $i=1\lc c$,
the flip \>$P_{2i-1}$ of the \>$(2i-1)$-st and \>$2i$-th tensor factors
of \,$L_{\bs\la}$ commutes with the \,$\gl_n\<$-action and thus acts on
\>$M_{\bs\la,\:\mu}$. Denote by $P_{\bs\la,\:\mu,\:c}\in\End(M_{\bs\la,\:\mu})$
the action of the product \,$P_1P_3\dots P_{2c-1}$ on \>$M_{\bs\la,\:\mu}$.

\vsk.1>
The operator \>$P_{\bs\la,\:\mu,\:c}$ is selfadjoint relative to the Hermitian
form \,$(\cdot,\cdot)_{\bs\la,\:\mu}$ on \>$M_{\bs\la,\:\mu}$ given
in Section \ref{shap sec}. Define a new Hermitian form
\,$(\cdot,\cdot)_{\bs\la,\:\mu,\:c}$ on \>$M_{\bs\la,\:\mu}$ by the rule:
for any $v,w\in M_{\bs\la,\:\mu}$,
\vvn.2>
\be
(v,w)_{\bs\la,\:\mu,\:c}\,=\,(P_{\bs\la,\:\mu,\:c}\>v,w)_{\bs\la,\:\mu}\,.
\vv.1>
\ee
Denote by \,$q(\bs\la,\mu,c)$ \,the signature of the form
\,$(\cdot,\cdot)_{\bs\la,\:\mu,\:c}$.

\begin{prop}
\label{charsign}
The signature \,$q(\bs\la,\mu,c)$ \>equals the coefficients of the monomial\\
$x_1^{\:\mu_1+n-1}x_2^{\:\mu_2+n-2}\?\dots x_n^{\:\mu_n}$ in the polynomial
\be
\Dl_n\cdot\>\prod_{i=1}^c\,S_{\la^{(2i)}}(x_1^2\?\lc x_n^2)
\prod_{j=2c+1}^k\!S_{\la^{(j)}}(x_1\lc x_n)\,.
\vv.2>
\ee
Here \,$\Dl_n$ is the Vandermonde determinant \Ref{Delta} and
\,$S_\la$ are Schur polynomials \Ref{Schur}.
\end{prop}
\begin{proof}
Since \,$P_{\bs\la,\:\mu,\:c}^2\:=1$, we have
\,$q(\bs\la,\mu,c)=\tr_{M_{\bs\la,\:\mu}}\!P_{\bs\la,\:\mu,\:c}$,
\,and the claim follows from Prop\-osition~\ref{char prop}.
\end{proof}

\begin{thm}
\label{sgn}
We have \,$d(\bs\la,\nu,\bs z)\ge|\>q(\bs\la,\mu,c)\:|$\,, \,where \>$\mu$ is
the complement of \>$\nu$ in the $n\times (d-n)$ rectangle,
\,$\mu\:=\:(d-n-\nu_n\>,\>d-n-\nu_{n-1}\>\lc\>d-n-\nu_1)$\>, cf.~\Ref{mu nu}.
\end{thm}
\begin{proof}
By Proposition \ref{sym bethe} and Lemma \ref{perm bethe}, the operators
\>$B_{ij}(\bs\la,\mu,\bs z)\in\End(M_{\bs\la,\:\mu})$ are selfadjoint relative
to the form \,$(\cdot,\cdot)_{\bs\la,\:\mu}^P$. By Corollary \ref{linalg2},
\be
\dim\Bigl(\,\bigcap_{i,j}\,\bigoplus_{\al\in\R}\,
M_{B_{ij}(\bs\la,\mu,\bs z)}(\al)\Bigr)\ge|\>q(\bs\la,\mu,c)\:|\,.
\vv-.1>
\ee
By Theorem \ref{langlands}, for any \,$\bs\la,\nu$ \>and generic complex
\,$\bs z$ the operators \>$B_{ij}(\bs\la,\mu,\bs z)$ \>are diagonalizable.
The same holds true under the reality condition on \,$\bs z,\bs\la$ \>imposed
in this section for any \>$c$\>. Thus for generic \,$\bs z$, the operators
\>$B_{ij}(\bs\la,\mu,\bs z)$ have at least \,$|\>q(\bs\la,\mu,c)\:|$ \,common
eigenvectors with real eigenvalues, which provides \,$|\>q(\bs\la,\mu,c)\:|$
distinct real points in \,$\Om(\bs\la,\nu,\bs z)$. Hence,
\,$d(\bs\la,\nu,\bs z)\ge|\>q(\bs\la,\mu,c)\:|$ \,for generic \,$\bs z$,
and therefore, for any \,$\bs z$\>, due to counting with multiplicities.
\end{proof}

\begin{cor}
\label{bound thm}
We have \,$d(\bs\la,\nu,\bs z)\geq |\>a(\bs\la,\nu,c)\:|$\>,
where \,$a(\bs\la,\nu, c)$ is the coefficient of the monomial
$x_1^{d-1-\nu_n} x_2^{d-2-\nu_{n-1}}\dots x_n^{d-n-\nu_1}$ in the polynomial
\be
\Dl_n\cdot\>\prod_{i=1}^c\,S_{\la^{(2i)}}(x_1^2\?\lc x_n^2)
\prod_{j=2c+1}^k\!S_{\la^{(j)}}(x_1\lc x_n)\,.
\vv.2>
\ee
Here \,$\Dl_n$ is the Vandermonde determinant \Ref{Delta} and
\,$S_\la$ are Schur polynomials \Ref{Schur}.
\end{cor}
\begin{proof}
The claim follows from Theorem \ref{sgn} and Proposition \ref{charsign}.
\end{proof}

Recall that the total number of points in \,$\Om(\bs\la,\nu,\bs z)$ equals
\,$\dim\:M_{\bs\la,\:\mu}\:=q(\bs\la,\mu,0)$. So if all points \,$z_1\lc z_k$
are real, Theorem \ref{sgn} claims that all points in
\,$\Om(\bs\la,\nu,\bs z)$ are real. It is proved in \cite{MTV4} that for
real \,$z_1\lc z_k$ all points in \,$\Om(\bs\la,\nu,\bs z)$ are real and
multiplicity-free. The proof of Theorem \ref{sgn} here is a modification
of the reasoning used in \cite{MTV4}.

\vsk.2>
Let \,$\labt$ \>be the \,$(k+1)$-tuple \,$\la^{(1)}\lc\la^{(k)},\:\nu$
\,and \,$\dl=(d-n\lc d-n)$ be the rectangular partition with \>$n$ \>rows.
There is a natural isomorphism of the multiplicity spaces \,$M_{\bs\la,\:\mu}$
and \,$M_{\labt,\:\dl}$ that is consistent with the forms
\,$(\cdot,\cdot)_{\bs\la,\:\mu}$ and \,$(\cdot,\cdot)_{\labt,\:\dl}$
and intertwines the operators \,$P_{\bs\la,\:\mu,\:c}$ \>and
\,$P_{\labt,\:\dl,\:c}$\>. Therefore, \,$q(\bs\la,\mu,c)=q(\labt,\dl,c)$
\,and \,$a(\bs\la,\nu,c)=a(\labt,\emptyset,c)$\>, where
\,$\emptyset=(0\lc 0)$ \>is the empty partition.

\vsk.1>
The corresponding statement in the osculating Schubert calculus is as follows.
Let \>$F$ be a M\"obius transformation mapping the real line to the real line
and such that \,$\infty\not\in\{F(z_1)\lc F(z_k),F(\infty)\:\}$.
Set \>$\zbti=\bigl(F(z_1)\lc F(z_k),F(\infty)\bigr)$. Then \>$F$ defines an
isomorphism of \,$\Om(\bs\la,\nu,\bs z)$ \>and \,$\Om(\labt,\emptyset,\zbti)$
that maps real points to real points, and
\,$d(\bs\la,\nu,\bs z)=d(\labt,\emptyset,\zbti)$\>.

\vsk.2>
Consider the transposed partitions \,$(\la^{(1)})'\?\lc(\la^{(k)})'\<,\nu'$,
and treat them as partitions with at most \,$d-n$ \>parts, adding extra zero
parts if necessary. Denote by \,$\bs\la'$ be the \,$k$-tuple
\,$(\la^{(1)})'\?\lc(\la^{(k)})'$. By the Lagrangian involution
for the osculating Schubert problems, see Section~4 of \cite{HSZ},
the intersections of Schubert varieties
\,$\Om(\bs\la,\:\nu,\,\bs z)\subset\Gr(n,d)$ \>and
\,$\Om(\bs\la'\?,\:\nu'\!,\,\bs z)\subset\Gr(d-n,d)$ are isomorphic by taking
the orthogonal complements in $\C_d[x]$ relative to the following bilinear
form: \,$\bra x^p\</p\:!\>,x^q\</q\:!\>\ket=(-1)^p\:\dl_{p+q,\:d-1}$\>,
\,$p=0\lc d-1$\>. In particular, \,$d(\bs\la,\nu,c)=d(\bs\la'\?,\nu'\!,c)$\>.

\vsk.2>
On the other hand, define the multiplicity space \,$M_{\bs\la'\!\<,\:\mu'}$
using the Lie algebra \,$\gl_{d-n}$. There is a natural isomorphism of the
spaces \,$M_{\bs\la,\:\mu}$ and \,$M_{\bs\la'\!\<,\:\mu'}$ that is consistent
with the forms \,$(\cdot,\cdot)_{\bs\la,\:\mu}$ and
\,$(\cdot,\cdot)_{\bs\la'\!\<,\:\mu'}$ and intertwines the operators
\,$P_{\bs\la,\:\mu,\:c}$ \>and \,$(-1)^m\:P_{\bs\la'\!\<,\:\mu'\?,\:c}$\>,
\vvn.1>
where \,$m=\sum_{i=1}^c|\:\la_{2\:i}|$. Therefore,
\,$q(\bs\la,\mu,c)=(-1)^m\:q(\bs\la'\?,\mu'\?,c)$ \,and
\,$a(\bs\la,\nu,c)=(-1)^m\:a(\bs\la'\?,\nu'\!,c)$\>.

\section{Comparison with the available results and data}
\label{compare sec}
In this section we will compare the lower bound for the number of real
solutions of the osculating Schubert problem provided by Corollary
\ref{bound thm} against other available data.

\vsk.2>
We discuss bounds that are independent of \,$z_1\lc z_k$ \,and say that
a bound is sharp if it is attained for some values of \,$z_1\lc z_k$. We
assume that the set \,$\{z_1\lc z_k\}$ \,is invariant under the complex
conjugation and \>$\la^{(i)}\?=\la^{(j)}$ whenever $z_i=\bar z_j$. The number
of complex conjugate pairs in \,$\{z_1\lc z_k\}$ \>is denoted by \>$c$\>.

To save writing, we will indicate only nonzero parts in partitions and
omit zeros. We call the osculating Schubert problem for the case of
\,$\la^{(1)}\<\lsym=\la^{(k)}\<=(\:1\:)$ and arbitrary $\nu$, the vector
Schubert problem.

The topological degree of a real Wronski map gives a lower bound for the number
of real solutions for the vector Schubert problem. This degree was computed
in \cite{EG2} and extended in \cite{SS} to the case of
\,$\la^{(1)}\<\lsym=\la^{(k-1)}\<=(\:1\:)$ and arbitrary $\la^{(k)}$ and $\nu$.
The result is given in terms of the sign\:-imbalance of the skew Young diagram
\,$\nu/\la^{(k)}$.
In the case \,$\la^{(k)}\<=(\:1\:)$ and $\nu=(m,m\lc m)$, where there are \>$p$
nonzero parts and \,$p\le m$\>, the sign\:-imbalance was computed in \cite{W}.
The results is \,$0$ \,for even \>$m+p$ \>and
\beq
\label{white}
\frac{(mp/2)!}{((m+p-1)/2)!}\;
\prod_{i=1}^{p-1}\,\frac{i!(m-i)!}{(m-p+2i)!((m-p-1)/2+i)!}
\eeq
for odd \,$m+p$\>. Unlike Corollary \ref{bound thm}, this bound
is independent on the number of complex conjugated pairs among \,$z_1\lc z_k$.

\vsk.2>
This bound is found to be not sharp for the case \>$m=p=3$, when \>$k=9$
\>and the problem is for $\Gr(3,6)$, in \cite{HSZ}. It is proved there that
the problem has at least two real solutions. For this case, Corollary
\ref{bound thm} gives lower bounds \>$a=42,0,2,0,6$ \>for \>$c=0,1,2,3,4$
respectively. Thus our bound is not sharp for $c=1,3$, but, according
to the computer data, see \cite{HS}, it is sharp for $c=0,2,4$.

\vsk.2>
On the other hand for the case of $p=3,m=5$, where \>$k=15$ and the problem is
for $\Gr(3,8)$, the topological bound of \cite{EG2} gives zero, the results of
\cite{HSZ} are not applicable, and Corollary \ref{bound thm} yields
\>$a=6006,858,198,42,6,10,10,70$ \>for \>$c=0,1,2,3,4,5,6,7$, respectively.
In particular, it shows that the real Wronski map
$\Gr^{\R}(3,8)\to\R{\Bbb P}^{15}$, which sends three\:-dimensional subspaces
of $\R_8[x]$ to their Wronski determinants, is surjective; see \cite{EG3}
for discussion of surjectivity of real Wronski maps.

\vsk.2>
In another example, \>$p=3\,,\;m=6$\>, that is, \>$k=18$, $\Gr(3,9)$\>,
the topological bound \Ref{white} is \,$12$\>, and Corollary \ref{bound thm}
gives: $a=87516,15444,3432,792,180,60,0,0,140,420$ \>for \,$c=0\lc 9$\>,
respectively. Thus the topological bound is better for $c=6,7$\>, while
Corollary \ref{bound thm} wins in the other cases.

\vsk.2>
For the case \>$p=2\,,\;c=m-1$\>, that is, \>$k=2\:m$, $\Gr(2,m)$\>,
the bounds of \Ref{white} and Corollary \ref{bound thm} coincide: both equal
zero for even \>$m$ \>and \,$(2s)\:!/(s\:!(s+1)\:!)$ \,for odd \>$m=2s-1$.
The bounds are known to be sharp in this case.

\vsk.2>
A large amount of computer generated data is available at \cite{HS},
so we have tested our bound against them. The bound given by Corollary
\ref{bound thm} coincides with the computer prediction in amazingly many cases.
For example, out of eleven computer generated tables presented in \cite{HHS},
the bound given by Corollary \ref{bound thm} is sharp in all cases except for
the second row of Table~5 corresponding to the vector Schubert problem with
\>$k=7$, \,$\nu=(3,3,3)$, \>for \,$\Gr(4,8)$. In this case, Corollary
\ref{bound thm} gives the bounds \>$a=20,0,4,0$ \>for \>$c=0,1,2,3$\>,
and the computer data are \,$20,{\bf 8},4,0$, indicating a possible deficiency
for \>$c=1$.

\vsk.1>
Also, for the case of \>$n=2$, there are sixty computer generated bounds with
nineteen of them being nonzero. All of them match the bounds given by Corollary
\ref{bound thm}.

\vsk.2>
Call the osculating Schubert problem symmetric if \,$\la^{(i)}\<=(\la^{(i)})'$
for all \,$i=1\lc k$\>, and \>$\nu=\nu'$. In this case, the numbers of real
solutions for different \>$c$ \>are often congruent modulo four, see
\cite{HSZ}. Since the number of real solutions for \>$c=0$ \>is known, it gives
under some additional assumptions a lower bound of two for the number of real
solutions whenever the number of complex solutions is not divisible by four.
It seems that many, though not all, discrepancies we found between the bound
given by Corollary \ref{bound thm} and the computer data happen in symmetric
problems. For example, the remark at the end of Section \ref{bound sec} shows
that \,$a(\bs\la,\nu,c)=0$ \>for the symmetric Schubert problem
if \,$\sum_{i=1}^c|\:\la_{2\:i}|$ \>is odd, but in some of those cases
the zero bound is not sharp according to the computer generated data.

\vsk.2>
Finally, consider the vector Schubert problem with \,$\nu=(k-n\lc k-n)$
\>having \>$n-1$ \>nonzero parts, for the Grassmannian \,$\Gr(n,k+1)$\,.
The number of real solutions of this problem for given $z_1,\dots,z_k$ has been found in \cite{HHS}
and is given by the coefficient \,$r(k,n,s)$ \>of the monomial $x^{k-n}y^{n-1}$
in the polynomial $(x+y)^{k-1-2s}(x^2+y^2)^s$, where \,$k-1-2s$ \>is the number
\vvn.1>
of real roots of the polynomial \,$g(u)=\frac d{du}\:\prod_{i=1}^k(u-z_i)$\>.

\vsk.1>
It is easy to check that \,$r(k,n,s-1)\ge r(k,n,s)$ \,if \,$1\le s<k/2$\>.
By Rolle's theorem, \,$s\le c$ \,if \>$2c<k$\>, and \,$s\le c-1$
\,if \>$2c=k$. Thus either \,$r(k,n,c)$ \>or \,$r(k,n,c-1)$ \,gives the lower
bound for the number of real solutions of the Schubert problem in question,
depending on whether \,$2c<k$ \>or \,$2c=k$. These lower bounds are sharp
because the equalities \,$s=c$ \,for \>$2c<k$ \>and \,$s=c-1$ \,for
\>$2c=k$ \,are attained as the following examples show.

\begin{example}
$s=c\,,\;2c<k$\>. For sufficiently small real \,$\eps$\>, the polynomial
\vvn.2>
\be
\prod_{i=1}^c\,(u^2\?+1-\eps^i)\,\prod_{j=1}^{k-2c}\,(u-\eps^j)
\vv.1>
\ee
has exactly \,$k-2c$ \>real roots and its derivative has exactly \,$k-1-2c$
\>real roots.
\end{example}
\begin{example}
$s=c-1\,,\;2c=k$\>. The polynomial \,\:$(x^2+1)^c$ \,has no real roots and
its derivative has exactly one real root.
\end{example}

For $n=3\,,\;k=14$\>, and \>$c=0\lc 7$\>, the sharp lower bounds respectively
equal \,$78,56,38,24,14,8,6,6$\>, while the bounds given by Corollary
\ref{bound thm} are \,$78,54,34,18,6,2,6,6$. Similarly,
for $n=4\,,\;k=11$\>, and \>$c=0\lc 5$\>, the sharp lower bounds are
\,$120,64,32,16,8,0$ \,versus the bounds \,$120,48,8,8,8,0$ \,given
by Corollary \ref{bound thm}.

\vsk>
\section*{Acknowledgments}
We are grateful to A.\,Eremenko for pointing out the reference \cite{P}.

\end{document}